\documentclass{amsproc}

\RequirePackage{algorithm}
\RequirePackage{algpseudocode}
\RequirePackage{amssymb}
\RequirePackage{graphicx}
\RequirePackage{fullpage}
\usepackage{comment}
\RequirePackage[colorlinks]{hyperref}
\RequirePackage{amsthm,amsmath}
\RequirePackage[numbers]{natbib}

\newtheorem{theorem}{Theorem}[section]
\newtheorem{lemma}[theorem]{Lemma}

\theoremstyle{definition}

\newtheorem{proposition}[theorem]{Proposition}
\newtheorem{corollary}[theorem]{Corollary}

\newcommand{\X}{\mathbf{X}}
\newcommand{\Z}{\mathbf{Z}}
\newcommand{\x}{\mathbf{x}}
\allowdisplaybreaks

\theoremstyle{remark}

\numberwithin{equation}{section}

\begin{document}

\title{Ergodicity and steady state analysis for Interference Queueing Networks}

\author{Sayan Banerjee}
\address{Department of Statistics and Operations Research, University of North Carolina, Chapel Hill}
\email{sayan@email.unc.edu}

\author{Abishek Sankararaman}
\address{Electrical Engineering and Computer Sciences Department, University of California, Berkeley}
\email{abishek@berkeley.edu}

\date{}



\begin{abstract}
We analyze an interacting queueing network on $\mathbb{Z}^d$ that was introduced in \cite{interf_q} as a model for wireless networks. We show that the marginals of the minimal stationary distribution have exponential tails. This is used to furnish asymptotics for the maximum steady state queue length in growing boxes around the origin. We also establish a decay of correlations which shows that the minimal stationary distribution is strongly mixing, and hence, ergodic with respect to translations on $\mathbb{Z}^d$.
\end{abstract}

\maketitle

\section{Introduction and  Model}
\label{sec:model}


In this paper, we consider the Interference Queueing Network model introduced in \cite{interf_q}. 
The model consists of an infinite collection of queues, each placed at a grid point of a $d$ dimensional grid $\mathbb{Z}^d$. 
Each queue has arrivals according to an independent Poisson process with intensity $\lambda$. The departures across queues are however coupled by the \emph{interference} they cause to each other, parametrized by a sequence  $\{a_i\}_{i \in \mathbb{Z}^d}$, where $a_i \geq 0$ and $a_i = a_{-i}$, for all $i \in \mathbb{Z}^d$ and $\sum_{i \in \mathbb{Z}^d}a_i < \infty$. For ease of exposition, and without loss of generality, we shall assume that $a_0 = 1$. 
The state of the network at time $t \in \mathbb{R}$ is encoded by the collection of processes $\{X_i(t)\}_{i \in \mathbb{Z}^d} \in \mathbb{N}_0^{\mathbb{Z}^d}$, where $X_i(t)$ denotes the queue length at site $i \in \mathbb{Z}^d$ at time $t$. Conditional on the queue lengths $\{X_i(t)\}_{i \in \mathbb{Z}^d}$, the departures across queues are independent with rate of departure from any queue $i \in \mathbb{Z}^d$ at time $t \in \mathbb{R}$ given by $\frac{X_i(t)}{\sum_{j \in \mathbb{Z}^d}a_j X_{i-j}(t)}$. \textcolor{black}{Here, and in the rest of the paper, we adopt the convention that $0/0 = 0$. }
Under these conditions, Proposition 4.1 of \cite{interf_q} gives that the process is well-defined in a path-wise sense, even when 	the interference sequence has infinite support, namely $a_i>0$ for infinitely many $i \in \mathbb{Z}^d$.
Thus, the evolution of the queues are coupled, in a translation-invariant fashion, where the service rate at a queue is lower if the queue lengths of its neighbors, as measured by the interference sequence $(a_i)_{i \in \mathbb{Z}^d}$, are larger.
\\

Formally, we work on a probability space containing the collection of processes $(\mathcal{A}_i,\mathcal{D}_i)_{i \in \mathbb{Z}^d}$, where $\{\mathcal{A}_i\}_{i \in \mathbb{Z}^d}$ are independent Poisson Point Processes (PPP) on $\mathbb{R}$ with intensity $\lambda$; and $\{\mathcal{D}_i\}_{i \in \mathbb{Z}^d}$ are independent PPP of unit intensity on $\mathbb{R} \times [0,1]$. For each $i \in \mathbb{Z}^d$, the epochs of $\mathcal{A}_i$ denote the instants of arrivals to queue $i$. Similarly, any atom of the process $\mathcal{D}_i$ of the form $(t,u) \in \mathbb{R}\times [0,1]$ denotes an event of potential departure from queue $i$; precisely, a departure occurs at time $t$ from queue $i$ if and only if \textcolor{black}{$u \le \frac{X_i(t)}{\sum_{j \in \mathbb{Z}^d}a_{j}X_{i-j}(t)}$}. Thus, the queue length process $(\{X_i(t)\}_{i \in \mathbb{Z}^d})_{t \in \mathbb{R}}$ is a factor of the driving sequences $(\mathcal{A}_i,\mathcal{D}_i)_{i \in \mathbb{Z}^d}$. A proof of existence of the process is given in \cite{interf_q}.
\\

This model was introduced in \cite{interf_q}, as a means to study the dynamics in large scale wireless networks \cite{sbd_tit}.
In two or three dimensions, this model has a physical interpretation of a wireless network. 
Each grid point (queue) represents a `region of geographical space',
and each customer represents a wireless link, i.e., a transmitter-receiver pair.
For analytical simplicity, the link length (the distance between transmitter and receiver) is assumed to be $0$, so that a single customer represents both a transmitter and receiver.
The stochastic system models the spatial birth-death dynamics of the wireless network, where links arrive randomly in space, with the transmitter having an independent file of exponentially distributed size that it wants to communicate to its receiver.
A link (customer) subsequently exits the network after the transmitter finishes transmitting the file to its receiver.
The duration for which a transmitter transmits (i.e., a customer stays in the network) is governed by the rate at which a transmitter can transmit the file.
As wireless is a shared medium, the rate of file transfer at a link depends on the geometry of nearby concurrently transmitting links \textemdash if there are a lot of links in the vicinity, i.e., large interference, the rate of file transfer is lowered. 
In our system, the instantaneous rate of file transfer at a link in queue $i \in \mathbb{Z}^d$ is equal to the Signal to Noise plus Interference Ratio $\frac{1}{\sum_{i \in \mathbb{Z}^d}a_i	x_{i-j}(t)}$. Here, all transmitters transmit at unit power which is received at its corresponding receiver without attenuation (numerator is $1$). However, the corresponding receiver also receives power from other neighboring transmitters that reduces the rate of transmission. The interfering power is attenuated through space, with the attenuation factor given by the interference sequence $\{a_i\}_{i \in \mathbb{Z}^d}$.
As there are $x_i(t)$ links in queue $i\in\mathbb{Z}^d$, and they all have independent file sizes, the total rate of departure at a queue is then $\frac{x_i(t)}{\sum_{j \in \mathbb{Z}^d}a_jx_{i-j}(t)}$.
We refer the reader to \cite{interf_q}, \cite{sbd_tit}, \cite{stolyar_conjecture} for more information on  the origin of this stochastic model and its applications to understanding wireless networks. \\

\textcolor{black}{Mathematically, this model lies at the interface between queueing networks and interacting particle systems. Most well known queueing networks with interactions between servers, like the Join-the-shortest-queue policy and Power-of-$d$-choices policy \cite{mitzenmacher2001power,van2018scalable}, incorporate global interactions between servers and the interaction between any two fixed servers approaches zero (in a suitable sense) as the system size increases. On the other hand, well known interacting particle systems like the exclusion process, zero range process, contact process, voter model, Ising model, etc., \cite{liggett2012interacting} have strong nearest neighbor interactions but they often have explicit stationary measures and/or locally compact state space (each site can take one of finitely many values/configurations). This model has nearest neighbor interactions as well as locally non-compact state space (queue lengths are unbounded), thus making many tools from either of the above two broad fields inapplicable. In particular, stationary measures, if they exist, are far from explicit and natural aspects of the stationary dynamics of the process, like uniqueness of stationary measure, decay of correlations, typical and extremal behavior of queue lengths, and convergence rates to stationarity from arbitrary initial configurations, are non-trivial to analyze and quantify. Moreover, the ratio-type functional dependence of the service rates on neighboring queues makes obtaining quantitative estimates  challenging and most of the analysis necessarily has to rely on `soft' arguments using qualitative traits of the model. Recently, motivated by this model, the first author revisited an interacting particle system called the Potlatch process \cite{liggett1981ergodic}, which shares many aspects in common with this model, but the simpler functional form of rates enables one (see \cite{banerjee2020rates}) to quantify rates of convergence (locally and globally) to equilibrium. Similar models have also appeared in the economics literature to analyze opinion dynamics on social networks \cite{acemouglu2013opinion}.}
\\


The paper \cite{interf_q} established stability criteria, namely that if $\lambda < \frac{1}{\sum_{j \in \mathbb{Z}^d}a_j}$, then there exists a translation invariant (on $\mathbb{Z}^d$) stationary distribution for the queue lengths. 
\textcolor{black}{The crucial property of the dynamics noted in \cite{interf_q} was the following form of monotonicity: if at time $t \in \mathbb{R}$, there are two initial configurations $\{X_i(t)\}_{i \in \mathbb{Z}^d}$ and $\{X_i^{'}(t)\}_{i \in \mathbb{Z}^d}$ such that for all $i \in \mathbb{Z}^d$, $X_i(t) \leq X_i^{'}(t)$ (assuming the system starts at time $t \in \mathbb{R}$), and if the processes $\{X_i(s) : i \in \mathbb{Z}^d, s \ge t\}$ and $\{X_i^{'}(s) : i \in \mathbb{Z}^d, s \ge t\}$ are constructed using the same arrival and departure PPP $(\mathcal{A}_i,\mathcal{D}_i)_{i \in \mathbb{Z}^d}$, then under this coupling, almost surely, for all $s \geq t$ and all $i \in \mathbb{Z}^d$, $X_i(s) \leq X_i^{'}(s)$. }
Monotonicity is then used to define the following notion of stability. For each $t \geq 0$ and $s \geq -t$, denote by $\{X_{i;t}(s)\}_{i \in \mathbb{Z}^d} \in \mathbb{N}_0^{\mathbb{Z}^d}$ the queue lengths at time $s$, when the system was started with all queues being empty at time $-t$, i.e., for all $i \in \mathbb{Z}^d$, $X_{i;t}(-t) = 0$. 
Monotonicity implies that under the above (synchronous) coupling, such that, almost surely, for all $s \in \mathbb{R}$, for all $i \in \mathbb{Z}^d$, the map $t \mapsto X_{i;t}(s)$ is non-decreasing.
The stationary version of the process is then defined as $\{X_{i;\infty}(s)\}_{i \in \mathbb{Z}^d}$, where for any $s \in \mathbb{R}$ and $i \in \mathbb{Z}^d$, $X_{i;\infty}(s) := \lim_{t\to\infty}X_{i;t}(s)$ in the almost sure sense. \textcolor{black}{It was shown in \cite{interf_q} (see Proposition 4.3 there) that $\{X_{i;\infty}(s)\}_{i \in \mathbb{Z}^d}$ is indeed a stationary solution to the dynamics which is \emph{minimal} in the sense that any other stationary solution stochastically dominates this solution in a coordinate-wise sense for all time. We will refer to this coupled `backward' construction of the process $\{X_{i;t}(s)\}_{i \in \mathbb{Z}^d} : s \ge -t\}$ (for $t \ge 0$), as well as $\{X_{i;\infty}(s)\}_{i \in \mathbb{Z}^d} : s \in \mathbb{R}\}$, as the ``Coupling-From-The-Past" (CFTP) construction.}
\\

 In the rest of the paper, we shall assume that $\lambda <  \frac{1}{\sum_{j \in \mathbb{Z}^d}a_j}$ and that the process $\{X_i(t) : i \in \mathbb{Z}^d, t \in \mathbb{R}\}$ is stationary and distributed according to the unique minimal stationary solution to the dynamics. 
Proposition 4.3 in \cite{interf_q} gives that for any $i \in \mathbb{Z}^d$ and $t \in \mathbb{R}$, the steady state queue length satisfies $\mathbb{E}[X_i(t)] = \frac{\lambda}{1-\lambda\sum_{j \in \mathbb{Z}^d}a_j}$. Subsequently, \cite{stolyar_conjecture} established that for all $\lambda <  \frac{1}{\sum_{j \in \mathbb{Z}^d}a_j}$, for all $i \in \mathbb{Z}^d$, $t \in \mathbb{R}$, $\mathbb{E}[(X_i(t))^2] < \infty$.\\

\textcolor{black}{In this paper, we show that the marginals of the minimal stationary distribution, in fact, has exponential tails (Theorem \ref{thm:exp_bound}). This is used to obtain asymptotics for the maximum queue length in steady state in growing boxes around the origin (Corollary \ref{max}). We further show a decay of correlations between the queue lengths of two sites as the distance between the sites increases (Theorem \ref{thm:corr_decay_general}). This, in turn, implies that the stationary distribution is strongly mixing, and thus, ergodic with respect to translations on $\mathbb{Z}^d$. An ergodic theorem is presented in Corollary \ref{meanET}.}

\section{Main Results}

\subsection{Exponential moments and stationary distribution tail bounds}
The first result concerns the existence of exponential moments for queue lengths which, in turn, yields two-sided exponential tail bounds on the marginals of the minimal stationary distribution.

\begin{theorem}
For all $\lambda < \frac{1}{\sum_{j \in \mathbb{Z}^d}a_j}$, there exists a constant $c_0 > 0$, such that for all $c \in [0,c_0)$, all $i \in \mathbb{Z}^d$ and $t \in \mathbb{R}$, 
\begin{equation}\label{expmom}
\mathbb{E}[e^{c X_i(t)}] < \infty.
\end{equation}
Moreover, for all $\lambda < \frac{1}{\sum_{j \in \mathbb{Z}^d}a_j}$, there exist constants $c_1,c_2,x_0 > 0$, such that, for all $x \geq x_0$, $i \in \mathbb{Z}^d$ and $t \in \mathbb{R}$,
	\begin{align}\label{probup}
	e^{-c_1x} \leq \mathbb{P}[X_i(t) \geq x] \leq e^{-c_2x}.
	\end{align}
	\label{thm:exp_bound}
\end{theorem}

The above theorem can be used to derive the following asymptotics for the maximum queue length in steady state in growing boxes around the origin.

\begin{corollary}\label{max}
For every $\lambda < \frac{1}{\sum_{j\in\mathbb{Z}^d}a_j}$, there exist positive constants $C_1, C_2$, such that for any $t \in \mathbb{R}$,
$$
\lim_{N \rightarrow \infty}\mathbb{P}\left(C_1 \log N \le \max_{i \in \mathbb{Z}^d : \|i\|_{\infty} \le N} X_i(t) \le C_2 \log N\right) =1.
$$
\end{corollary}

Theorem \ref{thm:exp_bound} and Corollary \ref{max} are proved in Section \ref{sec:proof_exponential_moment}.

\subsection{Correlation Decay and Mixing of the Stationary Queue Length Process}

%

The main result of this section shows that the stationary queue lengths at distinct sites show a decay of correlations in space as the distance between the sites increases. This, in fact, shows that the minimal stationary distribution is strongly mixing, and thus, ergodic with respect to translations on $\mathbb{Z}^d$. Before stating the results, we briefly recall some notions from ergodic theory.\\


Let $\{X_i\}_{i \in \mathbb{Z}^d} \in \mathbb{N}_0^{\mathbb{Z}^d}$ be a sample from the minimal stationary distribution of the dynamics. The law of $\X := \{X_i\}_{i \in \mathbb{Z}^d}$ induces a natural probability measure $\boldsymbol{\mu}$ on $\left(\mathbb{N}_0^{\mathbb{Z}^d}, \mathcal{B}\left(\mathbb{N}_0^{\mathbb{Z}^d}\right)\right)$ given by $\boldsymbol{\mu}(A) := \mathbb{P}\left(\X \in A\right), A \in \mathcal{B}\left(\mathbb{N}_0^{\mathbb{Z}^d}\right)$.
For any $n \in \mathbb{N}$, and $h \in \{1,\cdots,d\}$, define $n\mathbf{e}_h:= (\underbrace{0,\cdots,0}_{h-1\text{}},n,\underbrace{0,\cdots,0}_{d-h})$, namely the vector in $\mathbb{Z}^d$ of all $0$'s except the $h$th coordinate that takes value $n$. For $h \in \{1,\cdots, d\}$ and $i \in \mathbb{Z}^d$, let $\theta_h(i) := i + \mathbf{e}_h$ denote the unit translation map on $\mathbb{Z}^d$ along the $h$-th coordinate. Denote the associated transformation on $\mathbb{N}_0^{\mathbb{Z}^d}$ by $T_h(\x) := \x \circ \theta_h, \x := (x_i)_{i \in \mathbb{Z}^d} \in \mathbb{N}_0^{\mathbb{Z}^d}$, where $(\x \circ \theta_h)_i := x_{\theta_h(i)}, i \in \mathbb{Z}^d$. By the translation invariance of the dynamics, $\boldsymbol{\mu} \circ T_h^{-1} = \boldsymbol{\mu}$ for any $h \in \{1,\cdots d\}$.  For any $h \in  \{1,\cdots, d\}$, the quadruple $\mathcal{Q}_h :=\left(\mathbb{N}_0^{\mathbb{Z}^d}, \mathcal{B}\left(\mathbb{N}_0^{\mathbb{Z}^d}\right), \boldsymbol{\mu}, T_h\right)$ is referred to as a probability preserving transformation (ppt). Recall that $\mathcal{Q}_h$ is called \emph{strongly mixing} if for any $A, B \in \mathcal{B}\left(\mathbb{N}_0^{\mathbb{Z}^d}\right)$,
\begin{equation}\label{sm}
\lim_{n \rightarrow \infty} \boldsymbol{\mu}\left(A \cap T_h^{-n}B\right) = \boldsymbol{\mu}(A)\boldsymbol{\mu}(B),
\end{equation}
where, for $n \in \mathbb{N}$, $T_h^{-n}(\cdot)$ is the map on $\mathbb{N}_0^{\mathbb{Z}^d}$ obtained by composing $T_h^{-1}(\cdot)$ $n$ times. A set $A \in \mathcal{B}\left(\mathbb{N}_0^{\mathbb{Z}^d}\right)$ is called \emph{invariant} under the family of transformations $\{T_h\}_{h=1}^d$ if $T_h^{-1}A = A$ for all $h \in \{1,\cdots, d\}$. The family $\{T_h\}_{h=1}^d$ is called \emph{ergodic} if all invariant sets are trivial, that is, for any $A$ invariant, $\boldsymbol{\mu}(A) = 0 \text{ or } 1$. One can show that (see for eg. \cite{baccelli_bremaud}), if $\mathcal{Q}_h$ is strongly mixing for each $h \in \{1,\cdots, d\}$, then the family $\{T_h\}_{h=1}^d$ is ergodic.
\\

For any $K \in \mathbb{N}_0$, define $\X_{0,K} := \{X_i : i \in \mathbb{Z}^d, \|i\|_{\infty} \le K\}$, thought of as a random variable in $\mathbb{N}_0^{(2K+1)^d}$. Similarly, for $n \in \mathbb{N}$, $K \in \mathbb{N}_0$ and $h \in \{1,\cdots,d\}$, define $\X_{n\mathbf{e}_h,K} := \{X_i : i \in \mathbb{Z}^d, \|i - n\mathbf{e}_h\|_{\infty} \le K\}$.


\begin{theorem}
	Fix any $K \in \mathbb{N}_0$, $h \in \{1,\cdots,d\}$ and $0\leq \lambda < \frac{1}{\sum_{j \in \mathbb{Z}^d}a_j}$. Let $f,g$ be functions from $\mathbb{N}_0^{(2K+1)^d}$ to $\mathbb{R}$ such that $\mathbb{E}[f^2(\X_{0,K})] < \infty$ and $\mathbb{E}[g^2(\X_{0,K})]<\infty$.
The following limit exists:
	\begin{equation}\label{strongmix}
	\lim_{n \rightarrow \infty} \mathbb{E}[f(\X_{0,K}) g(\X_{n\mathbf{e}_h,K})] = \mathbb{E}[f(\X_{0,K})]\mathbb{E}[g(\X_{0,K})].
	\end{equation}
	In particular, $\mathcal{Q}_h$ is strongly mixing for any $h \in \{1,\cdots, d\}$. Hence, the family $\{T_h\}_{h=1}^d$ is ergodic.  
	\label{thm:corr_decay_general}
\end{theorem}

An immediate corollary is the following explicit formula for the asymptotic covariances of the stationary queue length processes.
\begin{corollary}
	\begin{align*}
	\lim_{n\to\infty} \mathbb{E}[X_0 X_{n\mathbf{e}_h}] = (\mathbb{E}[X_0])^2
	= \left( \frac{\lambda}{1-\lambda\sum_{j \in \mathbb{Z}^d}a_j}\right)^2.
	\end{align*}
	\label{lem:corr_decay_exists}
\end{corollary}
\begin{proof}
	Applying Theorem \ref{thm:corr_decay_general} with $K=0$ and $f(\ell)=g(\ell) = \ell, \ell \in \mathbb{N}$, and using Proposition 4.3 of \cite{interf_q}, yields this result.
\end{proof}
The ergodicity established in Theorem \ref{thm:corr_decay_general} directly implies the following version of the ergodic theorem. A sequence of finite subsets $\{F_r : r \in \mathbb{N}\}$ of $\mathbb{Z}^d$ with $\cup_{r \in \mathbb{N}} F_r = \mathbb{Z}^d$ is said to be a \emph{F{\o}lner sequence} if 
$$
\lim_{r \rightarrow \infty} \frac{|(\theta_h F_r) \Delta F_r|}{|F_r|} = 0,
$$
where $\theta_hF_r := \{f + \mathbf{e}_h : f \in F_r\}$ and $| \cdot |$ denotes set cardinality. The F{\o}lner sequence $\{F_r : r \in \mathbb{N}\}$ of $\mathbb{Z}^d$ is called \emph{tempered} if there exists $C>0$ such that for all $r \in \mathbb{N}$,
$$
\left|\bigcup_{u<r}F_u^{-1}F_r\right| \le C|F_r|.
$$
where $F_u^{-1}F_r = \{i - j : i \in F_r, j \in F_u\}$. For example, consider the sequence of boxes defined by
$$
B_r = \{ k = (k_1,\cdots, k_d) \in \mathbb{Z}^d : i^{(r)}_l \le k_l < j^{(r)}_l \text{ for all } 1 \le l \le d\}, \ i^{(r)},j^{(r)} \in \mathbb{Z}^d.
$$
If $B_r \subset B_{r+1}$ for all $r \in \mathbb{N}$ and $\cup_{r \in \mathbb{N}} B_r = \mathbb{Z}^d$, then $\{B_r : r \in \mathbb{N}\}$ forms a tempered F{\o}lner sequence with $C = 2^d$. The next corollary follows from ergodic theorems for amenable groups \cite[Theorems 1.1,1.3]{lindenstrauss1999pointwise} upon noting that $\mathbb{Z}^d$ is an amenable (additive) group. For $i \in \mathbb{Z}^d$ and $\x \in \mathbb{N}_0^{\mathbb{Z}^d}$, define $\x_{(i)} \in \mathbb{N}_0^{\mathbb{Z}^d}$ by $(\x_{(i)})_l := \x_{l+i}, l \in \mathbb{Z}^d$.

\begin{corollary}\label{meanET}
For any $f \in L^1(\mu)$ and any F{\o}lner sequence $\{F_r : r \in \mathbb{N}\}$, 
$$
\lim_{r \rightarrow \infty}\frac{1}{|F_r|}\sum_{i \in F_r}f\left(\X_{(i)}\right) = \int f(\x) \mu(d\x)
$$
in $L^1$ with respect to $\mu$. In addition, the above convergence holds almost surely with respect to $\mu$ if $\{F_r : r \in \mathbb{N}\}$ is a tempered F{\o}lner sequence (eg. $\{B_r : r \in \mathbb{N}\}$ defined above).
\end{corollary}

Theorem \ref{thm:corr_decay_general} is proved in Section \ref{sec:proof_corr_decay}.

\section{Proof of Theorem \ref{thm:exp_bound} and Corollary \ref{max}}
\label{sec:proof_exponential_moment}

\subsection{Proof of the Lower Bound in \eqref{probup}}\label{lowproof}


The proof of the lower bound in \eqref{probup} follows from a coupling argument, such that all queue lengths in the process are bounded from below by a collection of independent $M/M/1$ queues at each of the sites. 
We shall construct a coupling using the same driving sequence $(\mathcal{A}_i,\mathcal{D}_i)_{i \in \mathbb{Z}^d}$. For each $t \ge 0$, define the collection of independent $M/M/1$ queues $\{Q_{i;t}(s) : s \ge -t, i \in \mathbb{Z}^d\}$ such that for each $i \in \mathbb{Z}^d$, arrival happens at epochs of $\mathcal{A}_i$ and departure happens at epochs of $\mathcal{D}_i$ (that is, for any atom $(s', u')$ of $\mathcal{D}_i$, $s' \ge -t$, a departure happens from the queue $Q_{i;t}$ at time $s'$). 
The process $\{Q_{i;t}(s) : s \ge -t, i \in \mathbb{Z}^d\}$ is started empty at time $-t$, namely, almost-surely, $Q_{i;t}(-t) = 0$, for all $i \in \mathbb{Z}^d$.
Now, construct the process $\{X_{i;t}(s) : s \ge -t, i \in \mathbb{Z}^d\}$ using the same driving processes $(\mathcal{A}_i,\mathcal{D}_i)_{i \in \mathbb{Z}^d}$ using the CFTP recipe in Section \ref{sec:model}. Then, upon noting that for each $s \ge -t$, $\frac{X_i(s)}{\sum_{j \in \mathbb{Z}^d}a_j X_{i-j}(s)} \le 1$, it follows that under this coupling, almost surely, $Q_{i;t}(s) \le X_{i;t}(s)$ for all $s \ge -t$. For each $s \in \mathbb{R}$ and $i \in \mathbb{Z}^d$, by monotonicity of $Q_{i;t}(s)$ in $t$, which follows similarly as for $X_{i;t}(s)$, $Q_{i;t}(s)$ converges almost surely to a random variable $Q_{i, \infty}(s)$ as $t \rightarrow \infty$, where $Q_{i, \infty}(s)$ is distributed according to the stationary distribution of an $M/M/1$ queue with arrival rate $\lambda$ and departure rate $1$. Hence, by the coupling, for each $s \in \mathbb{R}$ and $i \in \mathbb{Z}^d$, $X_{i, \infty}(s)$ stochastically dominates $Q_{i, \infty}(s)$. The result now follows upon noting that the stationary queue length for a $M/M/1$ queue (with arrival rate less than departure rate) has exponential tails \cite{baccelli_bremaud}.

\subsection{Proof of \eqref{expmom} and the Upper Bound in \eqref{probup}}


We shall prove the exponential moment bound \eqref{expmom}, which automatically implies the upper bound in \eqref{probup}.
\\

 In order to proceed with the proof, we will need to set some notations. For any $L \in \mathbb{N}$, denote by $\{X_i^{L}(t) : t \in \mathbb{R}, i \in \mathbb{Z}^d\}$ the stationary process (distributed according to the minimal stationary regime) of the dynamics described in Section \ref{sec:model} constructed using a \emph{truncated} interference sequence $(a_i^{L} := a_i \mathbf{1}_{\| i \|_{\infty} \le L})_{i \in \mathbb{Z}^d}$.
 One can consider a natural coupling of all the processes $\{ (X_i^L(t),X_i(t)), t \in \mathbb{R}, i\in\mathbb{Z}^d, L \in \mathbb{N} \}$, where the lack of superscript denotes the dynamics without any truncation of the interference sequence, such that, {\em (i)} for each $L \in \mathbb{N}$, the process  $\{X_i^{L}(t) : t \in \mathbb{R}, i \in \mathbb{Z}^d\}$  is stationary, distributed according to its minimal stationary solution, {\em (ii)} the process $\{X_i(t) : t \in \mathbb{R}, i \in \mathbb{Z}^d\}$ is stationary distributed according to its minimal stationary solution, and {\em(iii)} almost-surely, for all $t \in \mathbb{R}$, $i \in \mathbb{Z}^d$ and $L_1 \leq L_2$, $X^{L_1}_i(t) \leq X_i^{L_2}(t) \leq X_i(t)$. This coupling can be constructed by using the same driving sequence $(\mathcal{A}_i,\mathcal{D}_i)_{i \in \mathbb{Z}^d}$ to construct all the processes involved using the CFTP recipe given in Section \ref{sec:model}.
 Moreover, from the proof of Proposition 4.3 of \cite{interf_q}, almost surely, for all $t \in \mathbb{R}$, and $i \in \mathbb{Z}^d$, $X_i^{L}(t)  \nearrow X_i(t)$ as $L \nearrow \infty$.
 \\

We will also need other modifications to the dynamics to restrict it on the torus, in order to carry out the proof. For any $n \in \mathbb{N}$, let $\mathbb{B}_n := [-n,n]^d \subset \mathbb{Z}^d$ denote the rectangle of side length $2n$ centered around the origin and let $\mathbb{T}_n$ denote the $n$-torus in $\mathbb{Z}^d$ seen as $\mathbb{B}_n$ with opposite faces identified.
 For any $L \in \mathbb{N}$, denote by $\{Y_i^{(n),L}(\cdot)\}_{i \in \mathbb{Z}_d}$ the process described in Section \ref{sec:model} whose dynamics are \emph{restricted} to the torus $\mathbb{T}_n$ as follows. 
For each $i \in \mathbb{B}_n$, the arrival process into the queue $Y_i^{(n),L}$ follows the PPP $\mathcal{A}_i$ as before, but departure happens at rate $\frac{Y_i^{(n),L}(t)}{\sum_{j \in \mathbb{Z}^d} a_{j}^{L} Y_{(i-j)\mod \mathbb{B}_n}^{(n),L}(t)}$ at time $t$. Set $Y_i^{(n),L}(\cdot) \equiv 0$ if $i \notin \mathbb{B}_n$.
  We define another process $(Z_i^{(n),L}(\cdot))_{i \in \mathbb{B}_n}$ to be the process described in Section \ref{sec:model}, but with the dynamics restricted to the set $\mathbb{B}_n$  with the edge effects. More precisely, set $Z_i^{(n),L}(\cdot) \equiv 0$ if $i \notin \mathbb{B}_n$, and for each $i \in \mathbb{B}_n$, the arrival process into the queue $Z_i^{(n),L}$ follows the PPP $\mathcal{A}_i$ as before, but departure happens at rate $\frac{Z_i^{(n),L}(t)}{\sum_{j \in \mathbb{Z}^d} a_{j}^{L} Z_{i-j}^{(n),L}(t)}$ at time $t$.\\
  

 From the monotonicity in the dynamics, there exists a coupling of all the processes (described as before using common PPP $\mathcal{A}_i$ and $\mathcal{D}_i$ respectively for arrival and departure) such that, {\color{black} they are all individually stationary} and almost surely, for all $n,L$, all $i \in \mathbb{B}_n$ and all $t \in \mathbb{R}$, $X_i^{L}(t) \geq Z_i^{(n),L}(t) $ and $Y_i^{(n),L}(t) \geq Z_i^{(n),L}(t)$. 
  \\
  
  	The main results established in \cite{interf_q} \textcolor{black}{(see Theorem 1.1, Theorem 5.2, and Remark 5.5 there)} show that if $\lambda < \frac{1}{\sum_{j \in \mathbb{Z}^d}a_j}$, then for all $L$ and $n > L$, the two processes $\{Y_i^{(n),L}(\cdot)\}_{i \in \mathbb{B}_n}$ and $\{Z_i^{(n),L}(\cdot)\}_{i \in \mathbb{B}_n}$ admit a unique stationary solution  and the process $\{X_i^{L}(\cdot)\}_{i \in \mathbb{Z}^d}$ has a non-trivial minimal stationary solution.
  	From monotonicity, one can construct a coupling of the processes $\{(X_i^{L}(\cdot), Y_i^{(n),L}(\cdot), Z_i^{(n),L}(\cdot)) : i \in \mathbb{Z}^d, \text{ and } n, L \in \mathbb{N}\}$ such that, they are all individually stationary (with $\{X_i^{L}(t)\}_{i \in \mathbb{Z}^d}$ having the minimal stationary distribution for every $t,L$) and, almost surely,
  	\begin{itemize}
  		\item For each fixed $L$ and $n > L$, $Y_i^{(n),L}(t) \geq Z_i^{(n),L}(t)$, for all $i \in \mathbb{B}_n$ and all $t \in \mathbb{R}$,
  		\item $n \rightarrow Z_0^{(n),L}(t)$ is non-decreasing, for all $t \in \mathbb{R}$,
  		\item $\lim_{n \rightarrow \infty}Z_0^{(n),L}(t) = X^{L}_0(t)$, for all $t \in \mathbb{R}$,
  		\item $L \rightarrow X_i^{L}(t)$ is non-decreasing and $X^{L}_i(t) \nearrow X_i(t)$ as $L \rightarrow \infty$, for all $i \in \mathbb{Z}^d$ and all $t \in \mathbb{R}$.
  	\end{itemize}
\textcolor{black}{The third property above follows from Proposition 7.3 of \cite{interf_q}. The fourth property follows from monotonicity and the proof of Proposition 4.3 of \cite{interf_q}.}
  In the rest of the proof, we shall assume that $\lambda < \frac{1}{\sum_{j \in \mathbb{Z}^d}a_j}$ and the processes $\{X_i^{L}(\cdot)\}_{i \in \mathbb{Z}^d}$, $\{Y_i^{(n),L}(\cdot)\}_{i \in \mathbb{B}_n}$ and $\{Z_i^{(n),L}(\cdot)\}_{i \in \mathbb{B}_n}$  are all individually stationary and satisfy the above properties. The following is the key technical result needed for the proof of \eqref{expmom}.

\begin{proposition}
	Let $\lambda < \frac{1}{\sum_{j \in \mathbb{Z}^d}a_j}$ and $c_0 > 0$ such that $c_0e^{c_0}  = \frac{\frac{1}{\sum_{j \in \mathbb{Z}^d}a_j}-\lambda}{\lambda+1} =: D$. Then for all $c \in [0,c_0)$, $L \in \mathbb{N}$ and $n > L$, $ \mathbb{E}[e^{c Y_0^{(n),L}(0)}]  \leq \frac{D}{D - ce^c} < \infty$. 
	\label{prop:torus_moment}
\end{proposition}

Before giving a proof of the above proposition, we shall see how this concludes the proof of the exponential moment bound \eqref{expmom}, and thus the upper bound in \eqref{probup}.

\begin{proof}[Proof of \eqref{expmom}]
	By the first property above, almost surely, for all $i \in \mathbb{B}_n, L \in \mathbb{N}$ and $t \in \mathbb{R}$,  we have, $Y_i^{(n),L}(t) \geq Z_i^{(n),L}(t)$. Thus, Proposition \ref{prop:torus_moment} implies that, for all $0 \leq c < c_0$,  $$\sup_{n > L} \mathbb{E}[e^{cZ_0^{(n),L}(t)}] \leq   \sup_{n > L} \mathbb{E}[e^{cY_0^{(n),L}(t)}] \le \frac{D}{D - ce^c}.$$
	As, almost surely, for any $t \in \mathbb{R}$, $L \in \mathbb{N}$, $n \rightarrow Z_0^{(n),L}(t)$ is non-decreasing and $\lim_{n \rightarrow \infty}Z_0^{(n),L}(t) = X^{L}_0(t)$, monotone convergence theorem establishes that, for all $c \in [0,c_0)$, $$\mathbb{E}[e^{cX_0^{L}(t)}] = \lim_{n \rightarrow \infty}\mathbb{E}[e^{cZ_0^{(n),L}(t)}] \leq \sup_{n > L} \mathbb{E}[e^{cZ_0^{(n),L}(t)}] \leq  \sup_{n > L}\mathbb{E}[e^{cY_0^{(n),L}(t)}] \leq \frac{D}{D-ce^c}.$$  
	Since, for each $L \in \mathbb{N}$, the process $\{X_i^{L}(\cdot) : i \in \mathbb{Z}^d\}$ is stationary and, almost surely, for any $t \in \mathbb{R}$, $i \in \mathbb{Z}^d$, 
	$X^{L}_i(t) \nearrow X_i(t)$ as $L \rightarrow \infty$, and $\sup_{L \in \mathbb{N}} \mathbb{E}[e^{cX_0^{L}(t)}] \leq \frac{D}{D-ce^c} < \infty$, yet another application of the monotone convergence theorem yields that $\mathbb{E}[e^{cX_0(t)}] = \lim_{L \rightarrow \infty}\mathbb{E}[e^{cX_0^{L}(t)}] \leq \frac{D}{D-ce^c} < \infty$.


\end{proof}

We set some notation and state two technical lemmas before proving Proposition \ref{prop:torus_moment}. 
We will fix a $L \in \mathbb{N}$ and drop the superscript $L$ notation to lighten the notational burden.
For each $n \in \mathbb{N}$ and $k \geq 1$, denote by $\mu_k^{(n)} := \mathbb{E}[(Y_0^{(n)}(0))^k]$, recalling that $\{Y_i^{(n)}(0)\}_{i \in \mathbb{B}_n}$ is distributed according to the stationary distribution of the process $\{Y_i^{(n)}(\cdot)\}_{i \in \mathbb{B}_n}$. Observe that \textcolor{black}{Theorem 5.2} from \cite{interf_q} immediately yields that for all $ n\in \mathbb{N}$ and all $k \geq 1$, $\mu_k^{(n)} < \infty$ since $\lambda < \frac{1}{\sum_{j \in \mathbb{Z}^d}a_j}$. We state two useful lemmas.

\begin{lemma}
	Let $(y_i)_{i \in \mathbb{B}_n}$ be any non-negative sequence of real numbers. For any $i \in \mathbb{B}_n$, define $\displaystyle{R_i := \frac{y_i}{\sum_{k \in \mathbb{Z}^d} a_k y_{(i-k) \mod \mathbb{B}_n}}}$. Then,
	for all $j \geq 1$,  
	\begin{equation*}
	\sum_{i \in \mathbb{B}_n} R_i y_i^j \geq \frac{1}{\sum_{k \in \mathbb{Z}^d} a_k} \sum_{i \in \mathbb{B}_n} y_i^j.
	\end{equation*}
	\label{lem:internal2}
\end{lemma}

\begin{lemma}
	For all $ n \in \mathbb{N}$ and $k \geq 1$, 
	\begin{equation}
	D(k+1) \mu_k^{(n)} \leq \sum_{j=0}^{k-1} {k+1 \choose j} \mu_j^{(n)},
	\label{eqn:moment_bounds}
	\end{equation}
	where $D$ is given in Proposition \ref{prop:torus_moment}.
	\label{lem:internal}
\end{lemma}

Before proving the above two lemmas, we use them to prove Proposition \ref{prop:torus_moment}.

\begin{proof}[Proof of Proposition \ref{prop:torus_moment}]
	Let $n \in \mathbb{N}$ be arbitrary and fixed. Let $c_0 > 0$ be such that $c_0 e^{c_0} = D$, where $D$ is defined in Proposition \ref{prop:torus_moment}, and fix any $0 \leq c < c_0$. Let $m \geq 1$ be arbitrary. For $k \ge 1$, by multiplying both sides of equation (\ref{eqn:moment_bounds}) by $\frac{c^k}{k !}$,
	\begin{align*}
	D (k+1) \mu_k^{(n)} \frac{c^k}{k!} \leq  \sum_{j=0}^{k-1} {k+1 \choose j} \mu_j^{(n)} \frac{c^k}{k!}.
	\end{align*}
	Simplifying, we obtain
	\begin{align*}
	D  \mu_k^{(n)} \frac{c^k}{k!} \leq  \sum_{j=0}^{k-1}  \frac{1}{j ! (k+1-j)!} c^k \mu_j^{(n)}.
	\end{align*}
	For $m \in \mathbb{N}$, summing both sides from $k=1$ through to $m$,
	\begin{align}
	D \sum_{k=1}^{m}   \mu_k^{(n)} \frac{c^k}{k!} &\leq \sum_{k=1}^{m}\sum_{j=0}^{k-1}  \frac{1}{j ! (k+1-j)!} c^k \mu_j^{(n)} \nonumber\\
	&\stackrel{(a)}{=} \sum_{j=0}^{m-1} \frac{\mu_j^{(n)}}{j!} \sum_{k = j+1}^{m} \frac{1}{(k+1-j)!} c^k \nonumber\\
	& \stackrel{(b)}{=} \sum_{j=0}^{m-1} \frac{\mu_j^{(n)}}{j!} \sum_{u=0}^{m-j-1} \frac{c^{u+j+1}}{(u+2)!} \nonumber\\
	& \leq \sum_{j=0}^{m-1} \frac{\mu_j^{(n)}}{j!} c^{j+1} \sum_{u=0}^{\infty} \frac{c^u}{(u+2)!} \nonumber\\
	& \leq c \sum_{j=0}^{m-1} \frac{c^j \mu_j^{(n)}}{j!} \sum_{u=0}^{\infty} \frac{c^u}{u!} \nonumber \\ 
	&\leq c \sum_{j=0}^{m} \frac{c^j \mu_j^{(n)}}{j!} e^c. \label{eqn:exp_mom}
	\end{align}
	Step $(a)$ follows from swapping the order of summations. Step $(b)$ follows from the substitution $u = k-j-1$. Define $ S_m^{(n)} := \sum_{j=0}^{m} \frac{c^j \mu_j^{(n)}}{j!} $. Observe that $\mu_k^{(n)} \geq 0$, for all $k \geq 0$ and $n \in \mathbb{N}$, and thus $S_m^{(n)}$ is non-decreasing in $m$ and the (possibly infinite) limit $\lim_{m \rightarrow \infty}S_m^{(n)}$ exists. The calculation in (\ref{eqn:exp_mom}) yields that $D(S_m^{(n)} - 1 )\leq ce^c S_m^{(n)}$, which on re-arranging yields that $S_m^{(n)} \leq \frac{D}{D - ce^c} < \infty$. Taking a limit in $m$, we see that $\lim_{m \rightarrow \infty} S_m^{(n)} \leq \frac{D}{D - ce^c} < \infty$.  Thus, from Taylor's expansion and monotone convergence theorem, we have that  $\mathbb{E}[e^{c Y_0^{(n)}}] = \lim_{m \rightarrow \infty} S_m^{(n)} \leq \frac{D}{D - ce^c} < \infty$. Since the bound does not depend on $n$, and $n \in \mathbb{N}$ and $c \in [0,c_0)$ are arbitrary, the proof is concluded.

\end{proof}

We now give the proof of Lemma \ref{lem:internal2}.

\begin{proof}[Proof of Lemma \ref{lem:internal2}]
	
	By a direct application of Cauchy-Schwartz inequality, we have 
	\begin{align*}
	\left(\sum_{i \in \mathbb{B}_n} y_i^j\right)^2 \leq \left(\sum_{i \in \mathbb{B}_n} R_i y_i^j \right) \left(\sum_{i \in \mathbb{B}_n} \frac{y_i^j}{R_i} \right),
	\end{align*}
	{\color{black} where recall that $0/0$ in the RHS is interpreted as $0$}. It thus suffices from the above bound to establish that $\sum_{i \in \mathbb{B}_n} \frac{y_i^j}{R_i}  \leq (\sum_{k \in \mathbb{Z}^d} a_k) \sum_{i \in \mathbb{B}_n}y_i^j$. We do this as follows.
 \begin{align*}
	\sum_{i \in \mathbb{B}_n} \frac{y_i^j}{R_i} & =   \sum_{i \in \mathbb{B}_n} y_i^{j-1}\left( \sum_{k \in \mathbb{Z}_d} a_{k} y_{(i-k) \mod \mathbb{B}_n} \right) \\
	& \stackrel{(a)}{\leq } \sum_{i \in \mathbb{B}_n} \sum_{k \in \mathbb{Z}^d } a_{k} \left( \frac{j-1}{j}y_i^j + \frac{1}{j}y_{(i-k) \mod \mathbb{B}_n}^j\right) \\
	&=  \frac{j-1}{j}\sum_{k \in \mathbb{Z}^d } a_{k} \sum_{i \in \mathbb{B}_n} y_i^j + \frac{1}{j}\sum_{k \in \mathbb{Z}^d} a_{k} \sum_{i \in \mathbb{B}_n} y_{(i-k) \mod \mathbb{B}_n}^j \\
	& \stackrel{(b)}{=}  \frac{j-1}{j} \sum_{k \in \mathbb{Z}^d } a_{k}\sum_{i \in \mathbb{B}_n} y_i^j + \frac{1}{j}\sum_{k \in \mathbb{Z}^d} a_{k} \sum_{i \in \mathbb{B}_n} y_{i }^j  \\
	&= \sum_{k \in \mathbb{Z}^d } a_{k}\sum_{i \in \mathbb{B}_n} y_i^j .
\end{align*}
	Step $(a)$ follows from Young's inequality that for any $a,b \geq 0$, we have $ab \leq \frac{a^p}{p} + \frac{b^q}{q}$, for any $p,q > 0$ such that $p^{-1} + q^{-1} = 1$. Thus, $a^{j-1}b \leq (j-1)\frac{a^j}{j} + \frac{b^j}{j}$, where we set $p = \frac{j}{j-1}$ and $q = j$. Inequality $(b)$ follows from the observation that, by translational symmetry of the torus, $ \sum_{i \in \mathbb{B}_n} y_{(i-k) \mod \mathbb{B}_n}^j  = \sum_{i \in \mathbb{B}_n} y_i^j$, for all $k \in \mathbb{Z}^d$. 
	
\end{proof}

We are now ready to prove Lemma \ref{lem:internal}.
\begin{proof}[Proof of Lemma \ref{lem:internal}]
	
\textcolor{black}{Let $\{Y^{(n)}_i\}_{i \in \mathbb{Z}^d}$ be a collection of random variables sampled from the (unique) stationary distribution of $\{Y_i^{(n)}(\cdot)\}_{i \in \mathbb{B}_n}$}. For brevity, we shall drop the $n$ superscript and write $Y_i := Y_i^{(n)}$ for all $i \in \mathbb{B}_n$, and $\mu_k^{(n)} = \mu_k$, for all $k \geq 0$, as $n$ is fixed throughout the proof. We apply the rate-conservation equation to the Lyapunov function $V(\mathbf{y}) = \sum_{i \in \mathbb{B}_n} (y_i)^{k+1}$, writing $\mathbf{y} := \{y_i\}_{i \in \mathbb{B}_n}$. Since $\{Y_i\}_{i \in \mathbb{B}_n}$ is stationary, we have that $\mathbb{E}\left(\mathcal{L}V(\mathbf{Y})\right) = 0$, where $\mathcal{L}$ is the generator of the continuous time Markov process corresponding to our dynamics. This in particular yields that 
	\begin{align*}
	0 &= \lambda \sum_{i \in \mathbb{B}_n} \mathbb{E}[ ( (Y_i+1)^{k+1} - Y_i^{k+1}) ]+ \sum_{i \in \mathbb{B}_n} \mathbb{E}[ R_i( (Y_i-1)^{k+1}-Y_i^{k+1})] \\
	&= \lambda \sum_{i \in \mathbb{B}_n} \sum_{j=0}^{k} {k+1 \choose j} \mathbb{E}[Y_i^j] + \sum_{i \in \mathbb{B}_n} \sum_{j =0}^{k}   {k+1 \choose j} \mathbb{E}[ R_i Y_i^j (-1)^{k+1-j}  ] \\
	& = \sum_{i \in \mathbb{B}_n} (k+1)( \lambda \mathbb{E}[Y_i^k] - \mathbb{E}[R_i Y_i^k]) + \sum_{i \in \mathbb{B}_n} \sum_{j = 0}^{k-1} {k+1 \choose j} \mathbb{E}[  ( \lambda + R_i(-1)^{k+1-j}  ) Y_i^j].
	\end{align*}
	Now, rearranging the above equality, we obtain
	\begin{align*}
	\sum_{i \in \mathbb{B}_n} (k+1)( - \lambda \mathbb{E}[Y_i^k] + \mathbb{E}[R_i Y_i^k])  &=  \sum_{i \in \mathbb{B}_n} \sum_{j = 0}^{k-1} {k+1 \choose j} \mathbb{E}[  ( \lambda + R_i(-1)^{k+1-j}  ) Y_i^j] \\
&\stackrel{(a)}{\leq} \sum_{i \in \mathbb{B}_n}(\lambda+1) \sum_{j = 0}^{k-1} {k+1 \choose j} \mathbb{E}[ Y_i^j].
	\end{align*}
where step $(a)$ follows from the fact that $0 \leq R_i \leq 1$ for all $i \in \mathbb{B}_n$. Now, applying Lemma \ref{lem:internal2} to the LHS above,
\begin{equation*}
\sum_{i \in \mathbb{B}_n} (k+1) \left( - \lambda + \frac{1}{\sum_{k \in \mathbb{Z}^d}a_k}\right) \mathbb{E}[Y_i^k] \leq \sum_{i \in \mathbb{B}_n}(\lambda+1) \sum_{j = 0}^{k-1} {k+1 \choose j} \mathbb{E}[ Y_i^j].
\end{equation*}
Rearranging the last display concludes the proof as, by translation invariance, for all $1\leq j \leq k+1$ and all $i \in \mathbb{B}_n$, we have $\mathbb{E}[Y_i^j] = \mathbb{E}[Y_0^j]$. 

\end{proof}

\subsection{Proof of Corollary \ref{max}} 
\begin{proof}
Recall from Subsection \ref{lowproof} the coupling of $\{X_{i;\infty}(\cdot)\}_{i \in \mathbb{Z}^d}$ with a collection of stationary independent $M/M/1$ queues $\{Q_{i; \infty}(\cdot)\}_{i \in \mathbb{Z}^d}$, each queue having Poisson arrivals with rate $\lambda$ and departures with rate $1$, such that, almost surely, $X_{i;\infty}(t) \ge Q_{i;\infty}(t)$ for all $i \in \mathbb{Z}^d$ and $t \in \mathbb{R}$. For each $i \in \mathbb{Z}^d$ and $t \in \mathbb{R}$, the distribution of $1 + Q_{i;\infty}(t)$ is $\operatorname{Geometric}(1-\lambda)$ \cite{baccelli_bremaud}. Thus, for any $C < d/\log(1/\lambda)$,
\begin{align}\label{lowmax}
\mathbb{P}\left(\max_{i \in \mathbb{Z}^d : \|i\|_{\infty} \le N}X_i(t) \ge \lfloor C \log N \rfloor\right) &\ge \mathbb{P}\left(\max_{i \in \mathbb{Z}^d : \|i\|_{\infty} \le N}Q_{i;\infty}(t) \ge \lfloor C \log N \rfloor\right)\nonumber\\
& \ge 1 - \left(1 - \lambda^{C \log N}\right)^{(2N+1)^d}\nonumber\\
& \ge 1- \exp\{-2^dN^{C\log \lambda + d}\}
\rightarrow 1, \ \text{ as } N \rightarrow \infty.
\end{align}
Recall the constant $c_2$ appearing in the the upper bound of \eqref{probup}. For any $C' > \frac{d}{c_2}$, using the union bound and the upper bound in \eqref{probup},
\begin{equation}\label{upmax}
\mathbb{P}\left(\max_{i \in \mathbb{Z}^d : \|i\|_{\infty} \le N}X_i(t) > C' \log N\right) \le (2N+1)^d e^{-c_2 C'\log N} = (2N+1)^d N^{-c_2C'} \rightarrow 0,  \ \text{ as } N \rightarrow \infty.
\end{equation}
The corollary now follows from \eqref{lowmax} and \eqref{upmax}.

\end{proof}

\section{Proof of Theorem \ref{thm:corr_decay_general}}
\label{sec:proof_corr_decay}

\begin{proof}
	For this proof, let $\{X_i\}_{i \in \mathbb{Z}^d} \in \mathbb{N}_0^{\mathbb{Z}^d}$ be a sample from the stationary solution of the dynamics. Fix $K \in \mathbb{N}_0$. From the symmetry in the dynamics it suffices to prove \eqref{strongmix} for $h=1$. Moreover, it suffices to consider $f, g$ non-negative. The general case follows upon separately considering the positive and negative parts of $f,g$.
	\\
	
	We first consider bounded $f(\cdot)$ and $g(\cdot)$.
	As before, we will proceed via a version of the dynamics with a truncated interference sequence. Consider a sequence $L_n$, such that for every $n \in \mathbb{N}$, $L_n \in \mathbb{N}$ and $\lim_{n \to \infty}L_n = \infty$ and $\lim_{n \to \infty}\frac{L_n}{n} = 0$. Moreover, assume $n \mapsto L_n$ and $n \mapsto  \lfloor \frac{n}{2} \rfloor - L_n$ are non-decreasing in $n$ for $n \ge 2$. One valid choice of $\{L_n\}_{n \in \mathbb{N}}$ is $L_n := \sqrt{ \lfloor \frac{n}{2} \rfloor}, n \ge 1$.
	As before, for each $n \in \mathbb{N}$, denote the truncated interference sequence by $\{a_i^{L_n}\}_{i \in \mathbb{Z}^d}$, where $a_i^{L_n} := a_i\mathbf{1}_{\|i\|_{\infty} \leq L_n }$. 
	\\
	

		Let $n_0 \in \mathbb{N}$ be such that $n \geq 2L_n+2K + 2$ for all $n \ge n_0$. For $n \ge n_0$, define $\mathcal{X}^{(n)}:= \{ (z_1,\cdots,z_d) \in \mathbb{Z}^d : \lfloor \frac{n}{2} \rfloor -L_n \leq z_1 \leq \lceil \frac{n}{2} \rceil + L_n \}$.
		 Consider the CFTP construction (with the truncated interference sequence $\{a_i^{L_n}\}_{i \in \mathbb{Z}^d}$) of the dynamics, where the infinite system was started with all queues being empty at a time $-t \leq 0$ (i.e., $t$ is positive). From this all empty state at time $-t$ in the past, the dynamics is run in forward time with no arrivals at sites in $\mathcal{X}^{(n)}$ and independent PP($\lambda$) arrivals at other sites, and with departure rates governed by the truncated interference sequence $\{a_i^{L_n}\}_{i \in \mathbb{Z}^d}$.
		  For any $i \in \mathbb{Z}^d$, $n \ge n_0$ and $t \geq 0$, denote by $X_i^{(n;t)}$ the queue length at site $i$ at time $0$ for this system.
 Monotonicity in the dynamics implies that for each $i \in \mathbb{Z}^d$, the map $t \mapsto X_i^{(n:t)}$ is non-decreasing and hence an almost sure limit $X_i^{(n)}:= \lim_{t \rightarrow \infty} X_i^{(n;t)}$ exists. 
 In other words, the random variable $X_i^{(n)}$ is defined to be the queue length at site $i$, at time $0$, in the stationary regime of the infinite dynamics, constructed with the truncated interference sequence $\{a_i^{L_n}\}_{i \in \mathbb{Z}^d}$, and with the queues at sites in the set $\mathcal{X}^{(n)}$ ``frozen" without activity with $0$ customers at all time. For $n \ge n_0$, write $\X^{(n)}_{0,K} := \{X^{(n)}_i : i \in \mathbb{Z}^d, \|i\|_{\infty} \le K\}$ and $\X^{(n)}_{n\mathbf{e}_1,K} := \{X^{(n)}_i : i \in \mathbb{Z}^d, \|i - n\mathbf{e}_1\|_{\infty} \le K\}$.
Also, recall $\X_{0,K} := \{X_i : i \in \mathbb{Z}^d, \|i\|_{\infty} \le K\}$ and $\X_{n\mathbf{e}_1,K} := \{X_i : i \in \mathbb{Z}^d, \|i - n\mathbf{e}_1\|_{\infty} \le K\}$. 
\\

	We now collect several useful properties of the random variables $\X^{(n)}_{0,K}$ and $\X^{(n)}_{n\mathbf{e}_1,K}$. Under the synchronous coupling (same arrival and departure PPP), almost surely, 
	\begin{enumerate}
		\item For each $n \ge n_0$, $\X^{(n)}_{0,K} \leq \X_{0,K}$ and $\X^{(n)}_{n\mathbf{e}_1,K} \leq \X_{n\mathbf{e}_1,K}$ (here `$\le$' denotes co-ordinate-wise ordering).
		\item The map $n \mapsto \X^{(n)}_{0,K}$, $n \ge n_0$, is non-decreasing and $\lim_{n \rightarrow \infty} \X^{(n)}_{0,K} = \X_{0,K}$.
		\item For all $n \ge n_0$, $\X^{(n)}_{0,K}$ and $\X^{(n)}_{n\mathbf{e}_1,K}$ are independent and identically distributed. 
	\end{enumerate}
	
	
	The first property and the first part of the second property follow from monotonicity of the dynamics. To verify the limit in the second property, first note by monotonicity that $\X_{0,K}^{(\infty)} := \lim_{n \rightarrow \infty} \X_{0,K}^{(n)}$ exists and, by property 1, 
\begin{equation}\label{int1}
\X_{0,K}^{(\infty)} \le \X_{0,K}.
\end{equation}
We will now argue the reverse inequality.
For $L \in \mathbb{N}$, let $\{X_i^{(n),L} : i \in \mathbb{Z}^d\}$ denote the queue lengths at time $0$ under the CFTP construction for the stationary dynamics with the same arrival process $\mathcal{A}_i$ at sites $i \notin \mathcal{X}^{(n)}$, zero arrivals at sites in $\mathcal{X}^{(n)}$, and departure rate governed by the truncated interference sequence $(a_i^{L} := a_i \mathbf{1}_{\| i \|_{\infty} \le L})_{i \in \mathbb{Z}^d}$. Denote by $\{Z_i^{(\lfloor \frac{n}{2} \rfloor -L_n), L} : i \in \mathbb{Z}^d\}$ the queue lengths at time $0$ under the CFTP construction for the stationary dynamics with the same arrival process $\mathcal{A}_i$ at sites $i \in \mathbb{Z}^d$ with $\|i\|_{\infty} < \lfloor \frac{n}{2} \rfloor -L_n$, zero arrivals outside this set of sites, and departure governed by the interference sequence  $(a_i^{L})_{i \in \mathbb{Z}^d}$. Finally, denote by $\{X_i^L : i \in \mathbb{Z}^d\}$ the stationary queue lengths at time zero under the CFTP construction of the dynamics with arrival process $\mathcal{A}_i$ for all $i \in \mathbb{Z}^d$ but departure governed by the interference sequence  $(a_i^{L})_{i \in \mathbb{Z}^d}$.
As before, let $\X_{0,K}^{(n),L}  := \{X^{(n),L}_i : i \in \mathbb{Z}^d, \|i\|_{\infty} \le K\}$, $\Z_{0,K}^{(\lfloor \frac{n}{2} \rfloor -L_n), L} :=  \{Z^{(\lfloor \frac{n}{2} \rfloor -L_n),L}_i : i \in \mathbb{Z}^d, \|i\|_{\infty} \le K\}$ and $\X_{0,K}^{L}  := \{X^{L}_i : i \in \mathbb{Z}^d, \|i\|_{\infty} \le K\}$. 

By monotonicity, $\Z_{0,K}^{(\lfloor \frac{n}{2} \rfloor -L_n), L} \le \X_{0,K}^{(n),L} \le \X_{0,K}^{L}$ for all $n \ge n_0$, and hence,
\begin{equation}\label{int2}
\Z^{(\infty),L}_{0,K} := \lim_{n \rightarrow \infty}\Z_{0,K}^{(\lfloor \frac{n}{2} \rfloor -L_n),L} \le \X_{0,K}^{(\infty),L} :=  \lim_{n \rightarrow \infty} \X_{0,K}^{(n),L} \le \X_{0,K}^{L}.
\end{equation}
Moreover, as $\lfloor \frac{n}{2} \rfloor -L_n \rightarrow \infty$ as $n \rightarrow \infty$, by Proposition 7.3 of \cite{interf_q}, $\Z^{(\infty),L}_{0,K} = \X_{0,K}^{L}$ and hence, by \eqref{int2}, for any $L \in \mathbb{N}$,
\begin{equation}\label{int3}
\X_{0,K}^{(\infty),L} = \X_{0,K}^{L}.
\end{equation}
Again, by monotonicity, $\X^{(n),L}_{0,K} \le \X^{(n)}_{0,K}$ for all $n$ such that $L_n \ge L$, and hence, $\X^{(\infty),L}_{0,K} \le \X^{(\infty)}_{0,K}$. Hence, from \eqref{int3}, for any $L \in \mathbb{N}$,
\begin{equation}\label{int4}
\X_{0,K}^{L} \le \X^{(\infty)}_{0,K}.
\end{equation}
Finally, from the proof of Proposition 4.3 in \cite{interf_q}, almost surely, for any $i \in \mathbb{Z}^d$, $\lim_{L \rightarrow \infty} X^{L}_i = X_i$ and hence, by \eqref{int4},
\begin{equation}\label{int5}
\X_{0,K} \le \X^{(\infty)}_{0,K}.
\end{equation}
The limit in property 2 above now follows from \eqref{int1} and \eqref{int5}.


	To obtain the third property, note that as $a_i^{L_n} = 0$ for all $\|i\|_{\infty} > L_n$, there are no interactions between queues on either side of the frozen queue(s). Thus, $\X_{0,K}^{(n)}$ and $\X_{n\mathbf{e}_1,K}^{(n)}$ are independent. The identical distribution follows from the symmetry of the sites in $\{ i \in \mathbb{Z}^d, \|i\|_{\infty} \le K\}$ and $\{i \in \mathbb{Z}^d, \|i - n\mathbf{e}_1\|_{\infty} \le K\}$ with respect to the set $\mathcal{X}^{(n)}$ and the fact that $a_i = a_{-i}$, for all $i \in \mathbb{Z}^d$. 
\\

We now proceed as follows:
	\begin{align}
	\mathbb{E}[f(\X_{0,K}) g(\X_{n\mathbf{e}_1,K})] - \mathbb{E}[f(\X^{(n)}_{0,K})]\mathbb{E}[g(\X^{(n)}_{n\mathbf{e}_1,K})]  &=	\mathbb{E}[f(\X_{0,K})g(\X_{n\mathbf{e}_1,K})] - \mathbb{E}[f(\X_{0,K}^{(n)})g(\X_{n\mathbf{e}_1,K}^{(n)})]  \nonumber \\ 
	& = \mathbb{E}[f(\X_{0,K})(g(\X_{n\mathbf{e}_1,K}) - g(\X_{n\mathbf{e}_1,K}^{(n)}))]\nonumber\\
	&\qquad + \mathbb{E}[g(\X_{n\mathbf{e}_1,K}^{(n)})(f(\X_{0,K}) - f(\X_{0,K}^{(n)}))], \nonumber \\
	&= \mathbb{E}[f(\X_{n\mathbf{e}_1,K})(g(\X_{0,K}) - g(\X_{0,K}^{(n)}))] \nonumber\\
	&\qquad + \mathbb{E}[g(\X_{n\mathbf{e}_1,K}^{(n)})(f(\X_{0,K}) - f(\X_{0,K}^{(n)}))]. \label{eqn:corr_inter}
	\end{align}
	The first equality follows since $\X_{0,K}^{(n)}$ and $\X_{n\mathbf{e}_1,K}^{(n)}$ are independent random variables. The second equality follows from adding and subtracting $\mathbb{E}\left(f(\X_{0,K})g(\X_{n\mathbf{e}_1,K}^{(n)})\right)$. The third equality follows as, by the symmetry of the sites in $\{ i \in \mathbb{Z}^d, \|i\|_{\infty} \le K\}$ and $\{i \in \mathbb{Z}^d, \|i - n\mathbf{e}_1\|_{\infty} \le K\}$ with respect to the set $\mathcal{X}^{(n)}$ and the fact that $a_i = a_{-i}$, for all $i \in \mathbb{Z}^d$, the law of $(\X_{0,K},\X_{n\mathbf{e}_1,K}, \X_{n\mathbf{e}_1,K}^{(n)})$ is the same as that of $(\X_{n\mathbf{e}_1,K},\X_{0,K}, \X_{0,K}^{(n)})$. 
	\\

	As both $f,g$ are bounded functions on $\mathbb{N}_0^{(2K+1)^d}$ and $X_i$ and $X^{(n)}_i$ are integer valued random variables for all $i \in \mathbb{Z}^d$, using properties $2$ and $3$ above, dominated convergence theorem yields $\lim_{n \rightarrow \infty} \mathbb{E}[f(\X_{0,K}^{(n)})] = \mathbb{E}[f(\X_{0,K})]$, 
	$$
	\lim_{n \rightarrow \infty} \mathbb{E}[g(\X_{n\mathbf{e}_1,K}^{(n)})] = \lim_{n \rightarrow \infty} \mathbb{E}[g(\X_{0,K}^{(n)})] = \mathbb{E}[g(\X_{0,K})].
	$$
and
	\begin{align*}
	\lim_{n \rightarrow \infty}  \mathbb{E}[f(\X_{n\mathbf{e}_1,K})(g(\X_{0,K}) - g(\X_{0,K}^{(n)}))] = 0 = \lim_{n \rightarrow \infty}\mathbb{E}[g(\X_{n\mathbf{e}_1,K}^{(n)})(f(\X_{0,K}) - f(\X_{0,K}^{(n)}))].
	\end{align*}
Using these limits in \eqref{eqn:corr_inter}, we obtain \eqref{strongmix} for all bounded functions $f$ and $g$.\\

Now, we consider general $f$ and $g$. For any $\epsilon>0$, there exist simple functions $f^{(\epsilon)}$ and $g^{(\epsilon)}$ such that $\mathbb{E}\left(f(\X_{0,K}) - f^{(\epsilon)}(\X_{0,K}) \right)^2 < \epsilon^2$ and $\mathbb{E}\left(g(\X_{0,K}) - g^{(\epsilon)}(\X_{0,K}) \right)^2 < \epsilon^2$. Now,
\begin{align}\label{gen1}
\left|\mathbb{E}\left(f(\X_{0,K})g(\X_{n\mathbf{e}_1,K})\right)\right. & \left. - \mathbb{E}\left(f(\X_{0,K})\right)\mathbb{E}\left(g(\X_{n\mathbf{e}_1,K})\right)\right|\nonumber\\
&\le \left|\mathbb{E}\left(f(\X_{0,K})g(\X_{n\mathbf{e}_1,K})\right) - \mathbb{E}\left(f^{(\epsilon)}(\X_{0,K})g^{(\epsilon)}(\X_{n\mathbf{e}_1,K})\right)\right|\nonumber\\
&\qquad + \left|\mathbb{E}\left(f^{(\epsilon)}(\X_{0,K})g^{(\epsilon)}(\X_{n\mathbf{e}_1,K})\right) - \mathbb{E}\left(f^{(\epsilon)}(\X_{0,K})\right)\mathbb{E}\left(g^{(\epsilon)}(\X_{n\mathbf{e}_1,K})\right)\right|\nonumber\\
&\qquad + \left|\mathbb{E}\left(f^{(\epsilon)}(\X_{0,K})\right)\mathbb{E}\left(g^{(\epsilon)}(\X_{n\mathbf{e}_1,K})\right) - \mathbb{E}\left(f(\X_{0,K})\right)\mathbb{E}\left(g(\X_{n\mathbf{e}_1,K})\right)\right|.
\end{align}
By triangle inequality, Cauchy-Schwartz inequality and translation invariance of the dynamics,
\begin{align}\label{gen2}
\big|\mathbb{E}\left(f(\X_{0,K})g(\X_{n\mathbf{e}_1,K})\right) & - \mathbb{E}\left(f^{(\epsilon)}(\X_{0,K})g^{(\epsilon)}(\X_{n\mathbf{e}_1,K})\right)\big|\nonumber\\
& + \left|\mathbb{E}\left(f^{(\epsilon)}(\X_{0,K})\right)\mathbb{E}\left(g^{(\epsilon)}(\X_{n\mathbf{e}_1,K})\right) - \mathbb{E}\left(f(\X_{0,K})\right)\mathbb{E}\left(g(\X_{n\mathbf{e}_1,K})\right)\right|\nonumber\\
&\le 2\sqrt{\mathbb{E}\left(f(\X_{0,K})\right)^2}\sqrt{\mathbb{E}\left(g(\X_{0,K}) - g^{(\epsilon)}(\X_{0,K})\right)^2}\nonumber\\
&\qquad + 2\sqrt{\mathbb{E}\left(g^{(\epsilon)}(\X_{0,K})\right)^2}\sqrt{\mathbb{E}\left(f(\X_{0,K}) - f^{(\epsilon)}(\X_{0,K})\right)^2}\nonumber\\
&\le 2\epsilon \left(\sqrt{\mathbb{E}\left(f(\X_{0,K})\right)^2} + \sqrt{\mathbb{E}\left(g(\X_{0,K})\right)^2} + \epsilon\right).
\end{align}
Moreover, as $f^{(\epsilon)}$ and $g^{(\epsilon)}$ are bounded,
\begin{equation}\label{gen3}
\lim_{n \rightarrow \infty} \left|\mathbb{E}\left(f^{(\epsilon)}(\X_{0,K})g^{(\epsilon)}(\X_{n\mathbf{e}_1,K})\right) - \mathbb{E}\left(f^{(\epsilon)}(\X_{0,K})\right)\mathbb{E}\left(g^{(\epsilon)}(\X_{n\mathbf{e}_1,K})\right)\right| = 0.
\end{equation}
Using \eqref{gen2} and \eqref{gen3} in \eqref{gen1}, we obtain
\begin{multline*}
\limsup_{n \rightarrow \infty}\left|\mathbb{E}\left(f(\X_{0,K})g(\X_{n\mathbf{e}_1,K})\right) - \mathbb{E}\left(f(\X_{0,K})\right)\mathbb{E}\left(g(\X_{n\mathbf{e}_1,K})\right)\right|\\
\le 2\epsilon \left(\sqrt{\mathbb{E}\left(f(\X_{0,K})\right)^2} + \sqrt{\mathbb{E}\left(g(\X_{0,K})\right)^2} + \epsilon\right).
\end{multline*}
As $\epsilon>0$ is arbitrary, this completes the proof of \eqref{strongmix}.\\

Take any $h \in \{1,\cdots, d\}$. Upon taking $f$ and $g$ to be indicator functions of cylinder sets $\mathcal{F}_0 \subset \mathcal{B}\left(\mathbb{N}_0^{\mathbb{Z}^d}\right)$, \eqref{strongmix} shows that \eqref{sm} holds for all $A, B \in \mathcal{F}_0$. A standard argument using the `good sets principle' can now be used to conclude that \eqref{sm} holds for all $A, B \in \mathcal{B}\left(\mathbb{N}_0^{\mathbb{Z}^d}\right)$. This shows that $\mathcal{Q}_h$ is strongly mixing for all $h \in \{1,\cdots, d\}$. Hence, $\{T_h\}_{h=1}^d$ is ergodic.

\end{proof}

\textbf{Acknowledgements} - AS thanks SB for hosting him at UNC Chapel Hill, where a large part of this work was done.  Most of this work was done when AS was a PhD student at UT Austin and he
thanks Fran\c cois Baccelli for supporting him through the Simons Foundation grant (\# $197892$) awarded to The University of Texas at Austin. SB was partially supported by a Junior Faculty Development Award made by UNC, Chapel Hill.

\bibliographystyle{plain}
\bibliography{interference-queues} 

\end{document}